\documentclass[12pt,a4paper]{amsart}
\usepackage[english]{babel}

\usepackage{amssymb} 
\usepackage{mathtools}
\usepackage{amsthm} 
\usepackage{graphicx} 
\usepackage[all]{xy} 
\usepackage{wasysym}
\usepackage{wrapfig} 
\usepackage{xcolor} 
\usepackage{verbatim}
\usepackage{hyphenat} 

\usepackage[hyphens]{url}

\usepackage{csquotes}
\usepackage{hyperref}
\hypersetup{
    final=true,
    plainpages=false,
    pdfstartview=FitV,
    pdftoolbar=true,
    pdfmenubar=true,
    bookmarksopen=true,
    bookmarksnumbered=true,
    breaklinks=true,
    linktocpage,
    colorlinks=true,
    linkcolor=teal,
    urlcolor=teal,
    citecolor=teal,
    anchorcolor=green
  }
\usepackage[hyphenbreaks]{breakurl}

\usepackage[style = alphabetic, sorting=nyt, url=true]{biblatex}
\DeclareFieldFormat[article]{volume}{\mkbibbold{#1}\addcolon \addspace}
\DeclareFieldFormat[article]{number}{\mkbibparens{#1}}
\renewbibmacro*{in:}{}

\DeclareMathOperator{\Ad}{Ad}

\DeclareMathOperator{\GL}{GL}
\DeclareMathOperator{\SO}{SO}
\DeclareMathOperator{\SU}{SU}

\DeclareMathOperator{\Sp}{Sp}
\DeclareMathOperator{\U}{U}

\DeclareMathOperator{\Spin}{Spin}

\title{The Ambrose-Singer Theorem for cohomogeneity one Riemannian manifolds}
\date{}

\author{J.~L. Carmona Jiménez, M. Castrillón López, J.~C. Díaz-Ramos}

\numberwithin{equation}{section}

\renewcommand{\to}{\longrightarrow}

\newcommand{\N}{\mathbb{N}}
\newcommand{\R}{{\mathbb{R}}}
\newcommand{\C}{{\mathbb{C}}}
\newcommand{\Z}{\mathbb{Z}}

\newcommand{\LM}{\mathcal{L}(M)}

\newcommand{\so}{\mathfrak{so}}

\newcommand{\T}{\nabla}
\newcommand{\TT}{\tilde{\T}}

\newcommand{\G}{\mathcal{G}}

\newcommand{\Isot}{\mathrm{Isot}}

\renewcommand{\a}{\mathfrak{a}}

\newcommand{\h}{\mathfrak{h}}
\newcommand{\m}{\mathfrak{m}}
\newcommand{\g}{\mathfrak{g}}
\newcommand{\gl}{\mathfrak{gl}}

\newcommand{\dt}{\frac{d}{dt} \Big |_{t=0}}

\newcommand{\SC}[1]{\underset{\text{\tiny{{$#1$}}}}{\text{\Large{$\mathfrak{S}$}}}}

\newcommand{\RH}{\mathbb{R}\mathrm{H}}

\newcommand{\curv}{\tilde{R}}
\newcommand{\tors}{\tilde{T}}

\newcommand{\pa}[1]{\frac{\partial}{\partial #1}}

\theoremstyle{definition}
\newtheorem*{AS-theorem}{Ambrose-Singer Theorem}
\newtheorem{definition}{Definition}[section]
\newtheorem{proposition}[definition]{Proposition}
\newtheorem{theorem}[definition]{Theorem}

\newtheorem{lemma}[definition]{Lemma}
\newtheorem{remark}[definition]{Remark}

\usepackage[a4paper]{geometry}
\begin{document}
\thanks{
The first and second authors have been supported by the project PID2021-126124NB-I00 (Spain). The third author has been supported by projects PID2022-138988NB-I00/AEI/10.13039/501100011033
(Spain) and ED431F 2020/04, ED431C 2023/31 (Xunta de Galicia, Spain).}

\begin{abstract}
        We characterize isometric actions whose principal orbits are hypersurfaces through the existence of a linear connection satisfying a set of covariant equations in the same spirit as the Ambrose-Singer Theorem for homogeneous space. These results are then used to describe isometric cohomogeneity one foliations in terms of such connections. Finally, we provide explicit examples of these objects in Euclidean spaces and real hyperbolic spaces.
\end{abstract}

\maketitle

{\footnotesize
\textbf{Key words.} Ambrose-Singer theorem, cohomogeneity one actions, canonical connection.

\textbf{MSC:} 53C05, 53C12, 53C40.
}

\section{Introduction}


In the early twentieth century, Élie Cartan gave a local characterization of locally symmetric spaces: a Riemannian manifold $(M,g)$ is locally symmetric if and only if its curvature tensor is invariant under parallel transport. This result was later generalized to Riemannian homogeneous manifolds, and it is now known as the Ambrose-Singer Theorem:
\begin{AS-theorem}\cite{AS1958}
	Let $(M,g)$ be a connected and simply-connected complete Riemannian manifold. Then, the following statements are equivalent:
	\begin{enumerate}
		\item The manifold $M$ is Riemannian homogeneous.
		\item The manifold $M$ admits a linear connection $\TT$ satisfying:
		\begin{equation*}
		\TT R = 0,\quad \TT S = 0,\quad \TT g = 0,
		\end{equation*}
	\end{enumerate}
where $R$ is the curvature tensor of the Levi-Civita connection $\T$ and  $S = \T - \TT$.
\end{AS-theorem} 


The explicit classification of the tensor $S$, known as the \emph{homogeneous structure tensor}, into $\mathrm{O}(n)$-irreducible classes, where $n = \mathrm{dim}(M)$, yields valuable geometric understanding. This classification benefits from the intricate interplay among partial differential equations, algebra, and geometry, as described by the Ambrose-Singer Theorem. This synergy was leveraged in the program that formally started with~\cite{TV1983}. It aims to understand Riemannian homogeneous manifolds through the lens of Riemannian homogeneous structures. For example, it characterizes a specific Riemannian homogeneous manifold by scrutinizing all its homogeneous structures (see~\cite{CGS2013, CC2022}). Additionally, different Riemannian transitive Lie group actions give rise to distinct Riemannian homogeneous structures. For a recent panoramic view of this program, see~\cite{CC2019}.


However, this program was exclusive to pseudo-Riemannian geometry. In fact, it has recently been understood that these techniques can be applied to non-metric frameworks, such as symplectic homogeneous manifolds, see~\cite{CC2022*}. In some sense, this opened the door to considering new generalizations of the Ambrose-Singer Theorem for non-transitive actions.


A cohomogeneity one action on a Riemannian manifold is an isometric action whose principal orbits are hypersurfaces.
If a group acts on a Riemannian manifold with cohomogeneity one, we say that such a manifold is a cohomogeneity one manifold.
These are precisely the objects of study in this article.
Techniques based on cohomogeneity one actions have successfully been used to find examples of interesting geometric structures, such as manifolds with positive curvature, Einstein metrics, or Ricci solitons, see~\cite{GWZ2008,GZ2000,W2006}.
The idea behind these constructions is the following.
Many geometric properties of Riemannian manifolds are encoded by partial differential equations whose solutions are usually very difficult to find.
However, if these structures are invariant under a cohomogeneity one action, these complicated PDEs translate into ordinary differential equations along a geodesic orthogonal to the orbits. This is a much easier problem, for which existence and uniqueness of solutions is well understood.
In general, the use of cohomogeneity one symmetries is a technique that traces back at least to the introduction of polar, cylindrical, or spherical coordinates in problems with rotational symmetries, such as the motion of planets revolving around the sun.


The aim of this article is to initiate a line of research, analogous to Tricerri and Vanhecke's program, adapted to cohomogeneity one Riemannian manifolds.  This objective is pursued by achieving the following goals.


Sec.~\ref{s:CO1-AS-Theorem} and Sec.~\ref{s:local-CO1-AS-Theorem} are devoted to proving an Ambrose-Singer Theorem for cohomogeneity one and locally cohomogeneity one Riemannian manifolds, respectively. We characterize connected, simply-connected, and complete cohomogeneity one Riemannian manifolds by the existence of a geodesically complete linear connection satisfying certain geometric covariant derivative equations. This connection is called the cohomogeneity one AS-connection. Analogously to the transitive framework, we introduce the difference tensor $S = \T - \TT$ where $\TT$ is the connection under consideration and $\T$ is the Levi-Civita connection. This is called the cohomogeneity one structure. Afterwards, we relax the topological conditions and locally characterize cohomogeneity one Riemannian manifolds, that is, we assume that there is a Lie pseudo-group of isometries acting on the manifold whose principal orbits are hypersurfaces. The existence of cohomogeneity one homogeneous structures describes the nature of being a locally cohomogeneity one manifold. 


Following the aforementioned program of~\cite{TV1983} or~\cite{CC2019}, together with the non-metric perspective (see \cite{CC2022*}), in Sec.~\ref{s:decom-CO1-AS}, we study the $\SO(n)$-irreducible submodules in the space of the tensor elements of cohomogeneity one structures. This decomposition yields four submodules for cohomogeneity one structures. In particular, one of these submodules encodes the second fundamental form of the leaves, and another contains the homogeneous structure tensor of the leaves. We can thus get valuable information of the cohomogeneity one Riemannian action by examining the projection onto these submodules.


In Sec.~\ref{s:difference-tensor}, we give a formula for the cohomogeneity one structure constructed in the main theorem of Sec.~\ref{s:CO1-AS-Theorem}. Afterwards, we study the uniqueness of this cohomogeneity one structure with respect to the action. Finally, Sec.~\ref{s:examples} is dedicated to presenting examples of cohomogeneity one structures in Euclidean spaces and real hyperbolic spaces.


\section{Preliminaries}

Recall that a homogeneous manifold $M=G/H$ is said to be reductive if the Lie algebra of $G$ can be decomposed as $\mathfrak{g}=\mathfrak{h}+\mathfrak{m}$ where $\mathfrak{h}$ is the Lie algebra of $H$ and $\mathfrak{m}$ is a subspace such that $\mathrm{Ad}_H\mathfrak{m}\subset \mathfrak{m}$. The following result is a generalization of the Ambrose-Singer Theorem with a non-necessarily Riemannian geometric structure.

\begin{theorem}\cite[Thm.~2.2]{CC2022*} Let $M$ be a connected, simply-connected manifold equipped with a tensor field $K$. Then there is a group $G$ of transformations preserving $K$ and acting reductively and transitively on $M$ (that is, $(M,K)$ is a reductive homogeneous manifold) if and only if there is a complete linear connection $\TT$ on $M$ such that
\begin{equation}\label{CO1.eq:AS-equations}
	\TT \curv = 0, \quad \TT \tors = 0, \quad \TT K = 0,
\end{equation}
where $\curv$ and $\tors$ are the curvature and torsion of $\TT$, respectively. 
\end{theorem}
The topological conditions on $M$ as well as the completeness of the canonical connection $\tilde\nabla$ are associated with the global nature of the action. When these conditions are weakened so that a manifold satisfies~\eqref{CO1.eq:AS-equations} only (that is, we have a so-called AS-manifold) there is a local version of the result as follows.
\begin{theorem}\cite[Thm.~3.7]{CC2022*}
	Let $M$ be a differentiable manifold with a geometric structure defined by a tensor $K$. Then the following assertions are equivalent:
	\begin{enumerate}
		\item The pair $(M,K)$ is a reductive locally homogeneous space, associated with a Lie pseudo-group $\G$ of local transformations of $M$ preserving $K$.
		\item  There exists a connection $\TT$ such that:
		\begin{equation*}
			\TT \curv = 0, \quad \TT \tors = 0, \quad \TT K = 0,
		\end{equation*}
		where $\curv$ and $\tors$ are the curvature and torsion of $\TT$, respectively.
	\end{enumerate}
\end{theorem}

In the proof of this last result, one shows that the Lie pseudo-group that is acting on $M$ is the Lie pseudo-group of local affine transvections of $\TT$. Recall that a transvection of a manifold endowed with a connection $(M, \TT)$ is a diffeomorphism $f\colon M \to M$ such that the \emph{natural lift} 
\begin{equation*}
	\begin{alignedat}{2}
		\tilde{f} \colon \LM &\to \LM \\
		u & \to f_{*,\pi(u)} \circ u.
	\end{alignedat}
\end{equation*}
to the frame bundle $\LM$ preserves every holonomy subbundle $P(u)\subset \LM$, i.~e., $\tilde{f}(P(u))\subset P(u)$ for all $u\in \LM$. The group $\mathrm{Tr}(M, \TT)$ of transvections is a subgroup of the group of affine transformations. Therefore, an affine transformation $f$ belongs to $\mathrm{Tr}(M, \TT)$ if and only if for every point $p \in M$ there is a piecewise smooth curve $\alpha$ joining $p$ with $f(p)$ such that the tangent map $f_* \colon T_pM \to T_{f(p)}M$ coincides with the parallel transport along $\alpha$.

\begin{theorem}{\cite[Thm.~I.25]{KK1980}}
	\label{kow} 
    Let $(M, \TT)$ be a connected manifold with an affine connection. Then the following two conditions are equivalent:
	\begin{itemize}
        \item The transvection group $\mathrm{Tr}(M, \TT)$ acts transitively on each holonomy bundle $P(u) \subset \LM$ and, in particular, it acts transitively on $M$.
        \item $M$ can be expressed as a reductive homogeneous space $G/H$ with respect to a reductive decomposition $\g = \h + \m$, where $G$ is effective on $M$, and $\TT$ is the canonical connection of $G/H$.
	\end{itemize}
    More precisely, if the first condition is satisfied, then we can write $M=G/H$ with $G=\mathrm{Tr}(M, \TT)$.
\end{theorem}

We deduce from the previous theorem that, on a reductive homogeneous manifold $M$, for any piecewise smooth curve $\alpha$ joining $p$ with $q$, there exists a global transvection $f$ such that 
$f(p) = q$
and the tangent map $f_* \colon T_pM \to T_qM$ coincides with the parallel transport along $\alpha$.

\section{Main result: global version}\label{s:CO1-AS-Theorem}

A Riemannian manifold $(M,g)$ is said to be a regular cohomogeneity one Riemannian manifold if there is a closed subgroup of isometries $G\subset \mathrm{Isom}(M,g)$ such that $M/G$ is connected and every orbit is principal of codimension one. We refer the reader to~\cite{BCO2016} for an introduction to cohomogeneity one actions.

\begin{lemma} 
    \label{CO1.lem:200}
    Let $(M,g)$ be a connected Riemannian manifold equipped with a vector field $\xi \neq 0$. Let $\TT$ be an affine connection such that, 
    \begin{equation*}
        \TT \xi = 0, \quad \TT_X g = 0,\quad \tors(X, Y) \in \mathcal{D}, \quad \forall\, X,\, Y  \in \mathcal{D},
    \end{equation*}
    where $\mathcal{D} = \{X :\: g(X,\xi) = 0 \}$ and $\tors$ is the torsion of $\TT$. Then, the distribution $\mathcal{D}$ is integrable, and the leaves are totally geodesic submanifolds with respect to $\TT$.
\end{lemma}

\begin{proof}
	We take two vector fields $X$, $Y$ such that $g(X,\xi) = 0$ and $g(Y, \xi) = 0$. Taking derivatives of these expressions, and using $\TT_X g = 0$ and $\TT \xi = 0$, we have that
	\begin{equation*}
		0 = g (\TT_Y X, \xi) + g(X, \TT_Y \xi)=g(\TT _Y X,\xi), \qquad 0 = g (\TT_X Y, \xi) + g(Y, \TT_X \xi)=g(\TT _X Y,\xi).
	\end{equation*} 
	Then $g([X,Y], \xi) = g(- \tors(X,Y) , \xi) = 0$, so that $\mathcal{D}$ is integrable. From the Frobenius Theorem, we have a foliation whose leaves are orthogonal to the vector field $\xi$. Finally, for any geodesic $\gamma (t)$ of $\tilde{\nabla}$, the function $g(\xi,\gamma'(t))$ is constant and the leaves are thus totally geodesic.
\end{proof}

\begin{theorem}\label{CO1.teorema principal}
	Let $(M,g)$ be a connected, simply-connected and complete Riemannian manifold. Then the following two are equivalent:
	\begin{enumerate}
		\item $(M,g)$ is a regular cohomogeneity one Riemannian manifold.
		\item There exists a complete linear connection $\TT$ and a vector field $\xi$ with $g (\xi, \xi) = 1$, such that,
		\begin{equation}\label{CO1.eq:ASCO1}
			\begin{alignedat}{2}
				\TT \curv &=  0, \quad \TT \tors = 0,  \quad \TT \xi = 0,\quad \\
				\TT_X g &= 0,\quad \tors(X, Y) \in \mathcal{D},\quad \forall\, X,\,Y  \in \mathcal{D},
			\end{alignedat}
		\end{equation}
        where $\mathcal{D} = \{X :\: g(X,\xi) = 0 \}$ is an integrable distribution whose leaves are embedded and closed, and $\curv$ and $\tors$ are the curvature and torsion of $\TT$.
	\end{enumerate}
\end{theorem}

\begin{proof} 
    \textbf{The proof of (1) implies (2).}  Let $G$ be a Lie group of isometries of $M$ acting on $M$ such that the orbits define a cohomogeneity one foliation. Let $p$ be a fixed point in $M$. As the orbits are principal, we can define a equivariant well-defined vector on the leaf $G\cdot p$ by
    \[
    \xi_{g\cdot p} = g \cdot \xi_p,
    \]
    where $\xi_p$ is a fixed unit normal vector at $p$, see~\cite[Sec.~2.1.8]{BCO2016}.

    Furthermore, it is well known that: 
 
    \begin{lemma}
        \label{lemma:orthogonal-geodesic}
        For any $p\in M$, the geodesic $\mathrm{exp}(t\xi_p)$ intersects all the leaves of the foliation orthogonally.
    \end{lemma}

    Now we consider the differentiable map $\varphi \colon \R \times (G \cdot p) \to M$ defined by
    \begin{equation*}
        \varphi_t (g \cdot p) = \exp _{g\cdot p}(t \xi_{g\cdot p}).
    \end{equation*}
    From the lemma above, this map is surjective. On the other hand, since $G$ sends geodesics to geodesics and $g \cdot \xi_q=\xi_{g\cdot q}$ because $G$-orbits are principal, we have that 
    \[
        g\cdot \varphi_t(q) = \varphi_t(g \cdot q), \qquad \forall q \in G\cdot p.
    \]
    In particular, since all orbits are principal (and thus, all isotropy groups are conjugate),
    \[
        \Isot_G (p) =  \Isot_G( \varphi_t (p)),
    \]
    where the isotropy $\Isot_G(q)$  of $q\in M$ is the Lie subgroup $\{ h \in G :\: h \cdot q  = q\}$.

    The vector field $\xi$ initially defined on $G\cdot p$ only, can now be extended to $M$ by simply taking derivatives of $\partial \varphi /\partial t$, $t\in \mathbb{R}$. This is well-defined since if we had $q=\varphi_{t_1}(p)=\varphi_{t_2}(g \cdot p)$ with 
    \begin{equation}
    \label{Co1.eq:3.3}
        \frac{\partial \varphi _t (p)}{\partial t}(t_1)=-\frac{\partial \varphi _t(g\cdot p)}{\partial t}(t_2),
     \end{equation}
    then
    \[
    	\varphi_{\frac{t_1+t_2}{2}} (p) = \varphi_{\frac{t_1+t_2}{2}} (g \cdot p) = g \cdot\varphi_{\frac{t_1+t_2}{2}} (p),
    \]
    that is, $g\in \mathrm{Isot}_G(\varphi_{\frac{t_1+t_2}{2}} (p))=\mathrm{Isot}_G(p)$. 
    But the uniqueness of geodesics would make~\eqref{Co1.eq:3.3} impossible for $g\cdot p=p$. 

	The orbits of $\xi$ are all of the same type. There is thus a group $A= \mathbb{R}$ or $\mathbb{S}^1$, and we have a transitive action on the left
	\begin{equation*}\label{Eq. acción global}
		\begin{alignedat}{2}
			(A \times G) \times M &\to M \\
			((\varphi_t , g), p) &\longmapsto \varphi_t (g\cdot p) = g \cdot \varphi_t (p).
		\end{alignedat}
	\end{equation*}

    Let $\mathrm{Isot}(p) = \{ (\varphi_t, g) \in A \times G:\: \varphi_t(g \cdot p) = p \}$ denote the isotropy of $p$ by the action of $A\times G$. 
    The subgroup $D=\{t\in A:\: \varphi_t(p)\in G\cdot p\}\subset A$ is closed and it is generated by the smallest $t_0$ such that $\varphi_{t_0} (p) = g\cdot p$. Therefore, $\mathrm{Isot}(p) = \mathcal{L}((\varphi_{t_0},g^{-1})) + \mathrm{Isot}_G (p)$, where $\mathcal{L}((\varphi_{t_0},g^{-1}))$ is the subgroup generated by $(\varphi_{t_0},g^{-1})$ in $A\times G$ and $\mathrm{Isot}_G (p)$ is the isotropy group of $p$ by the action of $G$. As $\mathcal{L}((\varphi_{t_0},g^{-1}))$ is discrete, the Lie algebra $\mathrm{Isot}(p)$ is the Lie algebra $\mathfrak{h}$ of $\mathrm{Isot}_G(p)$. We have that $[\h, \a] = 0$ because $G$ leaves $\xi$ invariant. 

    Since each orbit is a reductive homogeneous manifold, we have a connection in every leaf such that the geodesics of the connection are the curves $\exp (tX) \cdot p$ for every $X\in \m$ and $p\in M$. Moreover, there is a decomposition $\a + \g = \h + (\a + \m)$ and the subspace $\a +\m$ is $\Ad (\mathrm{Isot}_G (p))$-invariant. We have that 
    $\varphi_{t_0} \circ L_{g^{-1}}$ 
    is a global diffeomorphism preserving $\xi$ and the horizontal distribution defined by the connection above, i.~e., the subspace $\m$ is 
    $\Ad (\varphi_{t_0} \circ L_{g^{-1}})$-invariant. 
    It follows that $M$ is a reductive homogeneous manifold, and the canonical connection $\TT$  associated with the reductive decomposition above (see~\cite[Vol.~2, Ch.~X]{KN1963}) satisfies
	\begin{equation*}
		\TT \curv = 0, \quad \TT \tors = 0.
	\end{equation*}
	Since the Lie group $A \times G$ leaves $\xi$ invariant, $\TT\xi =0$ (see~\cite[p.~39, Prop.~1.4.15]{CC2019}).	

    To show that $\TT _X g = 0$ for all $X \in \mathcal{D}$, we consider, for $X \in \m$, the curve $\gamma(t) = \exp(t X) \cdot p$. The parallel transport along the curve $\gamma$ is the linear map given by the differential $(L_{\exp(tX)})_*\colon T_p M \to T_{\gamma(t)}M$ \cite[Ch.~X, Cor.~2.5]{KN1963}. As this map preserves the metric, we have that $\TT_{\gamma'} g = 0$. Finally, if we take $X$ and $Y$ as two vector fields tangent to the leaves, since $\TT_X g = 0$, we have $0 = g(\TT_XY, \xi) + g(Y, \TT_X \xi)=g(\TT_X Y,\xi)$. Hence the leaves are totally geodesic with respect to $\TT$ and therefore $ g(\tors (X,Y), \xi) = 0$.

\noindent \textbf{The proof of (2) implies (1).} 

    The completeness of $\TT$ together with conditions $\TT \tilde{R}=0$, $\TT \tilde{T}=0$ and $\TT \xi=0$ implies (see~\cite[p.~4, Thm.~2.2]{CC2022*})
    that $M$ is a homogeneous manifold $M = \bar{G}/\bar{H}$ (non-necessarily Riemannian), with reductive decomposition $\bar{\g} = \bar{\h} + \bar{\m} $, $\TT$ is the canonical connection and the vector field $\xi$ is $\bar{G}$-invariant. The Lie group $\bar{G}$ that acts transitively and effectively on $M = \bar{G}/\bar{H}$ is the Lie group of transvections of $\TT$ (see Theorem~\ref{kow}). Since $\TT_X g = 0$, for $X \in \mathcal{D}$, $g$ is invariant along parallel transport by curves orthogonal to $\xi$. 
	
	Let $\bar{\g}$ be the Lie algebra of $\bar{G}$ with reductive decomposition $\bar{\g} = \bar{\h} + \bar{\m}$. We can identify $\bar{\m}$ with $T_{e\bar{H}}M$ via, $\psi(X) = \dt \exp(tX) \cdot e\bar{H}$. As a consequence of this identification, we can endow $\bar{\m}$ with a Riemannian metric structure $g$ given by $g (X,Y) =  g|_{e\bar{H}} (\psi X, \psi Y)$. Then, we can consider the subspace 
    \[
    \m = \{ X\in \bar{\m} :\: g(X, \psi ^{-1} \xi) = 0\} \subset \bar{\m}
    \] 
    and the subalgebra $$\h = \mathrm{span} \{ [X,Y]_{\bar{\h}}:\: X,\, Y \in \m\} \subset \bar{\h}.$$  The algebra $\g =\h + \m$ is a Lie subalgebra of $\bar{\g}$, and $\h$ is a Lie subalgebra of $\bar{\h}$.	By~\cite[Vol.~2, p.~193, Cor.~2.5]{KN1963}, the geodesics of $\TT$ starting from $e\bar{H}$ have the expression $\exp(tX)\cdot e\bar{H}$ for all $X\in \m$ and all $t\in \R$. Indeed, parallel transport along such a geodesic is given by the differential map $(L_{\exp(tX)})_*$. Since $\TT_X g =0$, the transvection map $L_{\exp(tX)}$ is an isometry of $M$. From $[X,Y]_{\bar{\h}} = \curv_{XY}  = [\TT_X ,\TT_Y] - \TT_{[X,Y]}$ and $\TT_X g = 0$, we conclude $\curv_{XY} \cdot g = 0$. In particular, the transvection map $L_{\exp (tH)}$ associated with each $H \in \h$ is an isometry. Moreover, as $\curv_{XY} \cdot \xi = 0$ and $\curv_{XY} \cdot g = 0$, then $\m$ is an $\h$-module of $\g$ and every leaf is a reductive Riemannian homogeneous manifold.
	
	Let $G$ be the connected Lie subgroup of $\bar{G}$ associated with $\mathfrak{g}\subset \mathfrak{\tilde{g}}$. According to the above paragraph, $G$ is a group of isometries of $M$.
		
	\begin{lemma}
		$M$ is a regular cohomogeneity one manifold.
	\end{lemma}

    \begin{proof}
    	The Lie group $G$ preserves every leaf since its action preserves $g$ and $\xi$, and $G$ is connected. 
        Since $M$ is a connected homogeneous space, any two points $p$, $q \in M$ can be joined by a broken geodesic of $\TT$ such that each geodesic is an integral curve of $\xi$ or a $\TT$-geodesic in a leaf (we are using here the connectedness of $M$). Then, there exists  $L_g = L_{g_1} \circ \varphi_{t_1} \circ \ \dots \ \circ L_{g_m} \circ \varphi_{t_m}$ with $m\in \N$, $L_{g_i} \in G$ and $\varphi$ the flow of $\xi$ such that $L_g p = q$. As $\xi$ is invariant for every $L_{g_i}$, we have $L_g = \varphi_{t_1+ \ \dots \ +t_m} \circ L_{g_1} \circ \ \dots \ \circ L_{g_m}$. The integral curves of $\xi$ intersect all orbits. Therefore, to prove that $G$ acts transitively on every leaf, it is enough to prove this fact in only one leaf.
    		
        Let $N$ be the leaf of $e \bar{H}$. Since $N$ is a totally geodesic submanifold of $M$, by Lemma~\ref{CO1.lem:200}, every point $p\in N$ is connected to $e \bar{H}$ by a broken geodesic of $\TT$. As $G$ is connected and locally transitive on an open neighbourhood of $e \bar{H}$ (in the submanifold topology), $G$ acts transitively on $N = G/H$. As the orbits are closed, by~\cite[Thm.~5]{D2008}, we have that the closure $\tilde{G}$ of $G \subset \mathrm{Isom}(M)$ has the same orbits as $G$ and the action of $\tilde{G}$ is proper. 
        
        Finally, every orbit is principal since $\{ g \in \tilde{G}  :\: g \cdot e \bar{H} = e \bar{H}\} = \{ g \in \tilde{G} :\: g \cdot \varphi_t (e \bar{H}) =  \varphi_t (e \bar{H})\}$, for all $t\in \R$ and $\xi$ intersects all orbits. 
    \end{proof}

This finishes the proof of Theorem~\ref{CO1.teorema principal}.
\renewcommand*{\qedsymbol}{\(\blacksquare\)}\end{proof}
    

\begin{definition}
    A Riemannian manifold $(M,g)$ equipped with a unit vector field $\xi$ and a linear connection $\TT$ satisfying conditions~\eqref{CO1.eq:ASCO1} is called a \emph{cohomogeneity one Ambrose-Singer manifold} (CO1-AS-manifold for short). In this case, the $(1,2)$-tensor field $S = \T - \TT$ is called a \emph{cohomogeneity one structure}, where $\T$ is the Levi-Civita connection.
\end{definition}

\begin{remark}
    As it is well known, the same reductive homogeneous manifold may have different descriptions as a quotient $G/H$ as well as different reductive decompositions $\mathfrak{g}=\mathfrak{h}+\mathfrak{m}$. For instance, Euclidean spheres can be regarded as homogeneous manifolds exactly in the following ways (see~\cite{MS1943} and~\cite{AHL2023})
    \begin{equation*}
    \begin{alignedat}{111}
        \mathbb{S}^n &= \SO(n+1)/\SO(n), \quad \mathbb{S}^{2n+1}= \U(n+1)/\U(n), \\
        \mathbb{S}^{2n+1}&= \SU(n+1)/\SU(n),\\
        \mathbb{S}^{4n-1} &= \Sp(n)/\Sp(n-1), \quad \mathbb{S}^{4n-1}= \Sp(n+1)\Sp(1)/\Sp(n)\Sp(1),\\
        \mathbb{S}^{4n-1}&= \Sp(n)\U(1)/\Sp(n-1)\U(1),\\
        \mathbb{S}^6 &= G_2/\SU(3), \quad \mathbb{S}^{7}= \Spin(7)/\SU(3), \quad \mathbb{S}^{15}= \Spin(9)/\Spin(7).
    \end{alignedat}
    \end{equation*}
    Furthermore, for each of the quotients we may determine the set of all $G$-invariant metrics. Together with the choice of the complement $\mathfrak{m}$ to $\mathfrak{h}$, all these provide the collection of homogeneous tensors $S$.

    A similar situation happens in cohomogeneity one manifolds. For example, we might consider a warped product $M=\mathbb{R}\times_f \mathbb{S}^n$ for a smooth function $f:\mathbb{R}\to \mathbb{R}^+$ built with any of the previous metrics on the sphere. The action of the different groups $G$ obviously provides the same (trivial) foliation. 
\end{remark}

One advantage of AS or CO1-AS structures is that they can be used to distinguish different homogeneous or cohomogeneity one descriptions. With respect to CO1-AS structures, their classification is done up to equivalence in the following sense.

\begin{definition}
    Two regular cohomogeneity one manifolds $(M,g)$ and $(M',g')$ with the same group $G$ are called \emph{isomorphic} if and only if there exists a $G$-equivariant isometry $f \colon (M,g) \to (M',g')$.
\end{definition}

\begin{definition}
    Two cohomogeneity one structures $S$ in $(M,g)$ and $S'$ in $(M',g')$ are called \emph{isomorphic} if and only if there exists an isometry $f \colon (M,g) \to (M',g')$ that maps $S$ to $S'$.
\end{definition}

\begin{proposition} 
    Let $(M,g)$ and $(M',g')$ be two connected, simply-connected and complete $G$-regular cohomogeneity one manifolds with cohomogeneity one structures $S$ and $S'$, respectively. 
    If $S$ and $S'$ are isomorphic, then $(M,g)$ and $(M',g')$ are isomorphic.
\end{proposition}

\begin{proof}
    Let $f \colon (M,g) \to (M',g')$ be an isometry sending $S$ to $S'$. Then $f$ is an affine map from $\TT = \T - S$ to $\TT' = \T' - S'$, where $\T$ and $\T'$ are the Levi-Civita connections of $(M,g)$ and $(M',g')$, respectively. We now consider the Lie group $G$ generated by Theorem~\ref{CO1.teorema principal} applied to the first of the manifolds. It also acts on $M'$ through $f$. We know from Theorem~\ref{kow} that this Lie group is generated by the global transvections of the connection $\TT$. For every global transvection $F$ of $\TT$, $f \circ F \circ f^{-1}$ is a transvection map of $\TT'$. Therefore, $f$ is a $G$-equivariant map.
\end{proof}

\begin{proposition}
    Let $(M,g)$ and $(M',g')$ be two connected, simply-connected and complete $G$-regular cohomogeneity one manifolds. If $(M,g)$ and $(M',g')$ are isomorphic by $f \colon (M,g) \to (M',g')$ and $S$ is a cohomogeneity one structure for $(M,g)$, then $f_* S$ is a cohomogeneity one structure for $(M',g)$ that is isomorphic to $S$. 
\end{proposition}

\begin{proof}
   This proof is direct after the observation that $f$ is an affine diffeomorphism between the connections $\TT = \T -S$ and $\TT' = \T' - f_* S$, where $\T$ and $\T'$ are the Levi-Civita connections of $(M,g)$ and $(M',g')$, respectively. 
\end{proof}


\section{Main result: local version} \label{s:local-CO1-AS-Theorem}

Let $\mathcal{G}$ be a Lie pseudo-group of differentiable transformations on a manifold $M$ (we refer the reader to~\cite{S1992} or~\cite{A2021} for an exposition on this topic). Given a point $p_0\in M$, we define $\G(p_0)$ as the set of transformations for which $p_0$ belongs to the domain, and $\G(p_0,p_0) \subset \G (p_0)$ as the set of transformations $f$ such that $f (p_0)=p_0$. The quotient $H(p_0) = \G(p_0, p_0) / \sim$, with respect to the relation $f \sim f' \iff f |_U = f' |_U$ for some neighbourhood $U$ of $p_0$, is a Lie group (cf.~\cite[Ch.~1]{A2021}).
We now consider a frame $u_0\in \mathcal{L}(M)$ over the point $p_0$. We say that the action of $\G$ on $M$ is \emph{effective} and \emph{closed} if the map
\begin{equation}\label{CO1.eq:2}
	\begin{alignedat}{3}
		H(p_0) &\to  \GL(n,\R) \\
		f &\longmapsto u_0^{-1} \circ f_* \circ u_0
	\end{alignedat}
\end{equation}
is a monomorphism and its image $\mathbf{H}(u_0)$ is closed, i.~e.,  the image $\mathbf{H}(u_0)$ is a Lie subgroup of $\GL(n,\R)$. The morphism~\eqref{CO1.eq:2} will be called the isotropy representation of $\G$ on $M$. The effectiveness and closedness of this representation do not depend on the choice of $u_0$ (see~\cite[Prop.~3.4]{CC2022*}).
An effective and closed action of $\mathcal{G}$ on $M$ naturally induces an action of $\mathcal{G}$ on the frame bundle $\LM$:
\begin{equation*}
    (f,u) \mapsto f_{*,\pi(u)}\circ u, \quad u\in \mathcal{L}(M),
\end{equation*}
where $\mathbf{H}(u_0) \subset \GL(n,\R)$ represents the isotropy group at $u_0$ for the Lie groupoid defined by the germs of differentials of $\mathcal{G}$. It is called the \emph{linear isotropy group}.

Since two local isometries $f$ and $f'$ such that $f(p_0) = f'(p_0)$ coincide on a neighbourhood of $p_0$ (i.~e., $f \sim f'$) if and only if their differentials at $p_0$ are equal, the action of a pseudo-group of local isometries of a Riemannian manifold $(M,g)$ is always effective. 

If the action is both effective and transitive, given $u_0\in \mathcal{L}(M)$, the bundle
\begin{equation*}
    \mathcal{P}(u_0) = \{ f_* (u_0) :\: f \in \mathcal{G}\},
\end{equation*}
is a reduction of the frame bundle to the subgroup $\mathbf{H}(u_0)$. Under these conditions, for an element $\varphi \in H(p_0)$, we define
\begin{align*}
\Ad_{\varphi} \colon T_{u_0}P(u_0) &\to  T_{u_0}P(u_0)  \label{Equation RL3} \\
\dt (\varphi_t)_*(u_0) &\longmapsto \dt (\varphi \circ\varphi_t \circ \varphi^{-1})_* (u_0)
\end{align*}
where $\varphi _t \in \mathcal{G}$,
and $t$ belongs to a certain interval $(-\varepsilon , \varepsilon)$.

Consider a Lie pseudo-group $\mathcal{G}$ whose action is transitive, effective and closed. In this context, we say that the action is \emph{reductive} 
(see~\cite[Def.~3.5]{CC2022*})
if and only if the tangent space at $u_0$ admits a decomposition $T_{u_0}P(u_0) = \h + \m$, where $\h$ is the Lie algebra associated with $\mathbf{H}(u_0)$ and $\m $ is a $\Ad(H(p_0))$-invariant subspace.

\begin{theorem}{\cite[3.1.13]{CC2019}} \label{Loc Hom => Reductive}
    Let $(M,g)$ be a Riemannian manifold such that a pseudo-group $\mathcal{G}\subset \mathrm{Isom}_{loc} (M,g)$ is acting transitively. Then, the action of $\mathcal{G}$ is effective, closed and reductive.
\end{theorem}

\begin{definition}
    A Riemannian manifold $(M,g)$ is \emph{locally cohomogeneity one} if there exists a pseudo-group of local isometries $\mathcal{G}$ acting on $M$ in such a way that $M/\mathcal{G}$ is connected and every orbit is an embedded and closed submanifold of codimension one.

    A locally cohomogeneity one Riemannian manifold $(M,g)$ is said to be \emph{regular} if for any two points $p$ and $q$ of $M$, the Lie groups $H(p)$ and $H(q)$ are conjugate in $\G$, that is, for every $f\in H(p)$, there is $g\in \G$ such that $g^{-1} \circ f \circ g \in H(q)$. When $\G$ is a Lie group instead of a Lie pseudo-group, these conditions are equivalent, by \cite[Thm.~5]{D2008}, to saying that the action is proper for the closure of $\G$ in $\mathrm{Isom}(M)$ and that all orbits are principal.
\end{definition}

\begin{theorem} \label{CO1.teorema local}
	Let $(M,\, g)$ be a connected Riemannian manifold. The following two are equivalent:
	\begin{enumerate}
		\item $(M,g)$ is a regular locally cohomogeneity one Riemannian manifold.
		\item There exists a connection $\TT$ and a unit vector field $\xi\in\mathfrak{X}(M)$ such that
		\begin{equation}\label{CO1.eq:5.1}
			\begin{alignedat}{2}
				\TT \curv &=  0, \quad \TT \tors = 0,  \quad \TT \xi = 0,\quad \\
				\TT_X g &= 0,\quad \tors(X, Y) \in \mathcal{D}, \quad &\forall\, X,Y  \in \mathcal{D},
			\end{alignedat}
		\end{equation}
        where $\mathcal{D} = \{X\in\mathfrak{X}(M) :\: g(X,\xi) = 0 \}$, and $\curv$ and $\tors$ are the curvature and torsion of $\TT$, respectively. Furthermore, the maximal integral leaves of the distribution $\mathcal{D}$ (which is integrable according to Lemma~\ref{CO1.lem:200}) are embedded and closed.
	\end{enumerate}
\end{theorem}

\begin{proof}
    Although the structure of the proof is close to that of Theorem~\ref{CO1.teorema principal}, it requires substantial modifications as we give below.	
        
    \noindent \textbf{Proof of (2) implies (1).} From~\cite[Thm.~3.7]{CC2022*}, 
    we know that there exists a reductive Lie pseudo-group $\bar{\mathcal{G}}$ that acts transitively on $(M,g)$ and leaves $\xi$ invariant. Moreover, this pseudo-group $\bar{\mathcal{G}}$ can be identified with the Lie pseudo-group of local transvections of $\TT$. We thus consider the subset $\mathcal{G}$ of $\bar{\mathcal{G}}$ consisting of local transvections associated with parallel transports by curves in the leaves. Since $\TT \xi = 0$ and $\TT_X g = 0$ for all $X\in \mathcal{D}$, these local transvections preserve the vector field $\xi$, the Riemannian tensor $g$, and thus the codimension one distribution $\mathcal{D}$. 
	
	\begin{lemma}
		The subset $\mathcal{G}$ is a Lie pseudo-group.
	\end{lemma}
	
	\begin{proof}
		This follows directly from the definition of local transvection, i.~e., it is a local diffeomorphism that preserves every holonomy bundle.
	\end{proof}

    \begin{lemma}
        For any $f\in \mathcal{G}$ and any point $q\in \mathrm{dom}(f)$, both $q$ and $f(q)$ belong to the same leaf. Furthermore, the action of $\mathcal{G}$ on every leaf is transitive.
    \end{lemma}
    \begin{proof}
        The local transvection $f\in \mathcal{G}$ is associated with a $\TT$-parallel transport along a broken geodesic in one leaf $O$, and that leaf is thus invariant. However, the point $q\in\mathrm{dom}(f)$ does not necessarily belong to that leaf. In that case, there exists $t_0 \in \R$ such that $\varphi_{t_0} (q) \in O$, where $\varphi_t$ is the flow of $\xi$. The vector field $\xi$ is geodesic  for $\TT$, the differential of the flow is the $\TT$-parallel transport along $\xi$, and the flow sends leaves to leaves. As $f$ preserves $\xi$, the points $q$ and $f(q)$ belong to the same leaf $\varphi_{-t_0} (O)$.

        The transitivity comes from the existence of (broken) geodesics connecting any two points of any leaf.
    \end{proof}

\begin{lemma}
    Let $H(p)$ and $H(q)$ be the isotropy groups of $p$, $q\in M$. Then, $H(p)$ and $H(q)$ are conjugate in $\G$.
\end{lemma}
\begin{proof}
    Since $\TT \xi = 0$, the flow of $\xi$ commutes with any $f\in \G$. As $\G$ acts transitively, there exist $h \in \G$ and $t_0\in \R$ such that $p = \varphi_{t_0} \circ h (q)$, where $\varphi_t$ is the flow of $\xi$. Therefore, for every element $g\in H(q)$, we have that 
    $h^{-1} \circ \varphi_{-t_0} \circ g \circ \varphi_{t_0} \circ h \in H(p)$. Since $\varphi$ commutes with $\G$, $h^{-1} \circ g \circ h \in H(p)$. Conversely, for every $f \in H(p)$ we have $h \circ f \circ h^{-1} \in H(q)$. Therefore, $h^{-1} \circ H(q) \circ h = H(p)$.
\end{proof}
        
\noindent \textbf{Proof of (1) implies (2).} The procedure is to find a connection $\TT$ satisfying the first row of~\eqref{CO1.eq:5.1}. Again, 
from~\cite[Thm.~3.7]{CC2022*},
it is sufficient to show that there exists a Lie pseudo-group $\bar{\mathcal{G}}$ whose action is transitive (Lemma~\ref{L:transitively}), effective (Lemma~\ref{L:effectively}), closed (Lemma~\ref{L:closed}), and reductive (Lemma~\ref{L:reductive}) on the leaves. Afterwards, we prove the second row of~\eqref{CO1.eq:5.1}.

\begin{lemma}\label{lemmurcillo}
    Let $(M,g)$ be a regular locally cohomogeneity one Riemannian manifold associated with the Lie group action $\G$. Then, for any $p\in M$ and any unit vector $\xi_p$ orthogonal to $T_p (\G\cdot p)$, we have
    $$
    H(p) \cdot \xi_p = \xi_p,
    $$
    where $H(p)$ is given in \eqref{CO1.eq:2}.
\end{lemma}
\begin{proof}
    By contradiction, we assume that there exists $f\in \G$ such that $f(p) = p$ and $f_* \xi_p = -\xi_p$. As $\G \cdot p$ is embedded, there is an $\varepsilon >0$ small enough such that the curve
    $$\Sigma = \{ \exp_p (t\xi_p) :\: t \in (-\varepsilon, \varepsilon) \}$$
    intersects $\G\cdot p$ only at $p$.
    
    We take the points $q = \exp_p (\frac{\varepsilon}{4} \xi_p)$ and $\hat{q} = \exp_p ( - \frac{\varepsilon}{4} \xi_p)$ with $q \neq \hat{q}$. Thus $f(q) = \hat{q}$, that is, $f\notin H(q)$.
    Since the Lie groups $H(p)$ and $H(q)$ are conjugate, there exists $g\in \G$ such that $g \circ f \circ g^{-1} = h \in H(q)$ and $h \notin H(p)$. Moreover, by construction, $h_* \xi_q = -\xi_q$, thus, $h(p) = \exp_p (\frac{\varepsilon}{2} \xi_p) \in \Sigma$, which contradicts the fact that $\Sigma$ intersects $\G\cdot p$ only at $p$. 
\end{proof}
 
Due to Lemma~\ref{lemmurcillo}, any orbit admits a well-defined unit normal vector field $\xi$. As $M/\G$ is connected, in analogy with Theorem~\ref{CO1.teorema principal}, the vector field $\xi$ is globally defined on $M$. That is, for every orbit $\mathcal{G} \cdot p_0$ there exists an $\varepsilon > 0$ such that the flow $\varphi \colon (-\varepsilon, \varepsilon) \times U \to M$, where $U$ is an open set containing $\mathcal{G} \cdot p_0$, is defined. Let $\mathcal{G}$ be the Riemannian pseudo-group acting with cohomogeneity one on $M$. We define the pseudo-group $\bar{\mathcal{G}}$ generated by $\mathcal{G}$ and $\mathcal{A}$, the latter being the pseudo-group generated by flows of unit local vector fields orthogonal to the leaves.


\begin{lemma} \label{L:transitively}
	The Lie pseudo-group $\bar{\mathcal{G}}$ acts transitively on $M$.
\end{lemma}

\begin{proof}
    Let $p$ and $q$ be two points. If $p$ and $q$ belong to the same leaf, then there exists a local transformation $f\in \G$ such that $f(p)=q$. In general, it is easy to see that $A=\{ x \in M :\: f(p)=x, \, f\in \bar{\mathcal{G}}\}$ is open and closed. Since $M$ is connected, we get the result.
\end{proof}

As $\xi$ is invariant under the isometries of $\mathcal{G}$, every geodesic flow orthogonal to the leaves commutes with the local isometries of $\mathcal{G}$ by Lemma~\ref{lemma:orthogonal-geodesic}. Actually, every transformation $\bar{f} \in\bar{\mathcal{G}}$ can be written as $\bar{f} = \varphi \circ f$, where $\varphi \in \mathcal{A}$ and $f\in \G$.


\begin{lemma} \label{L:effectively}
    The Lie pseudo-group $\bar{\mathcal{G}}$ is effective.
\end{lemma}

\begin{proof}
    According to the first paragraph of Section~\ref{s:local-CO1-AS-Theorem}, let $H_{\G} (p_0) = \G (p_0, p_0)/\sim$ be the quotient space of elements $f \in \mathcal{G}$ fixing a point $p_0$ with respect to the relation $f \sim f' \iff f |_U = f' |_U$ for some neighbourhood $U$ of $p_0$. Analogously, the space of elements $\bar{f} \in\bar{\mathcal{G}}$ fixing a point $p_0$ is given by
	\begin{equation*}
		H(p_0) = H_{\G} (p_0) + \mathcal{L}\left(\varphi_{t_0} \circ f^{-1}\right),
	\end{equation*}
    where $\mathcal{L}(\varphi_{t_0} \circ f^{-1})$ is the discrete Lie group generated by the element $\varphi_{t_0} \circ f^{-1}$, and $t_0$ is the smallest $t_0>0$ such that $\varphi_{t_0} (p_0) = f(p)$ for some $f\in\mathcal{G}$.
    Note that the set $D = \{\varphi_{k\,t_0}(p_0) :\: k \in \Z \}$ is the intersection of $\{ \varphi_t (p_0) :\: t\in\R \}$ and $\mathcal{G}\cdot p_0$, and both subsets are closed in $M$. 

    The action $\mathcal{G}$ is effective because it is a Lie pseudo-group of isometries. 
    We show that if $(\varphi_{t_0} \circ f^{-1})_* (p_0) = Id_{T_{p_0}M}$, then $\varphi_{t_0} \circ f^{-1}$ is the identity map in an open neighbourhood $U$ of $M$. 
    There is an open neighbourhood $U$ such that both the domain and the image of $\varphi_{t_0} \circ f^{-1}$ are contained in $U$.  This map is $\mathcal{G}(U)$-equivariant by local isometries $h \in \mathcal{G}$ with domain $U$, i.~e., for every $h\in \mathcal{G}(U)$ there exits $h'\in \mathcal{G}(U)$ such that $\varphi_{t_0} \circ f^{-1} \circ h = h' \circ \varphi_{t_0} \circ f^{-1}$. For every local Killing vector field $X^*$, its flow comes from  a family $F_t$ of isometries with $t\in I$, where $I$ is a closed interval, and $F_t (p_0)$ is an integral curve of $X^*$. Due to the equivariance, there exist a local Killing vector field $Y^*$ and a family $H_t$ of isometries such that,
    \begin{equation*}
        \varphi_{t_0} \circ f^{-1} \circ F_t = H_t \circ \varphi_{t_0} \circ f^{-1}.
    \end{equation*}
    When applied to $p_0$, we get
    \begin{equation*}
        \varphi_{t_0} \circ f^{-1} \circ F_t (p_0) = H_t (p_0),
    \end{equation*}
    i.~e., $(\varphi_{t_0} \circ f ^{-1})_* (X^*(p_0)) = Y^* (p_0)$. Necessarily, $X^* = Y^*$; therefore $H_t = F_t$. Consequently, by applying this argument for every local Killing vector field, we conclude that  $(\varphi_{t_0} \circ f^{-1})$ fixes integral curves of Killing vector fields with the initial point $p_0$. Under these conditions, $\varphi_{t_0} \circ f^{-1} = Id_U$.      
\end{proof}


\begin{lemma} \label{L:closed}
	The action of the Lie pseudo-group $\bar{\mathcal{G}}$ is closed.
\end{lemma}

\begin{proof}
    The image of $H(p_0)$ by the isotropy map~\eqref{CO1.eq:2} is
    \begin{equation*}
        \mathbf{H}(u_0) = \mathbf{H}_{\G} (u_0) + \mathcal{L}(u_0{}^{-1} \circ (\varphi_{t_0} \circ f^{-1})_* \circ u_0),
    \end{equation*}
    where $\mathbf{H}_{\G} (u_0)$ is the image of $H_{\G} (p_0)$.
    On the one hand, as every leaf $\mathcal{G} \cdot p$ is a locally homogeneous Riemannian manifold and its action is reductive (see~\cite{T1992}), then $\mathbf{H}_{\mathcal{G}} (u_0)$ is equal to the holonomy group of one AS-connection. Consequently, it is closed in $\GL(n,\R)$.

    On the other hand, let $A = u_0{}^{-1} \circ (\varphi_{t_0} \circ f^{-1})_* \circ u_0 \in \GL(n,\R)$. Then, the space $\mathcal{L}(A)$ consists of the powers $A^k$ for $k\in \Z$. The matrix $A \in \GL(n,\R) \subset \GL(n,\C)$ diagonalizes as $A = P\cdot  D\cdot P^{-1}$. We consider a convergent sequence $\{A^{k_n}\}$ in $\GL(n,\C)$. Its eigenvalues cannot have a norm different from 1; otherwise, the limit would have a singular eigenvalue. If the eigenvalues are of the form $e^{iv}$ with $v$ irrational, $\{A^{k_n}\}$ would not converge. Then, the exponents of the eigenvalues are rational and the group $\mathcal{L}(A)$ is cyclic, finite, and closed. Finally, since $\mathbf{H}_{\mathcal{G}} (u_0)$ and $\mathcal{L}(A)$ are closed, it follows that $\mathbf{H} (u_0)$ is closed.
\end{proof}

As $\mathcal{L}((\varphi_{t_0} \circ g^{-1}))$ is discrete, the Lie algebra of $H(p_0)$ is equal to the Lie algebra $\h$ of $\mathrm{H}_{\G}(p_0)$.


\begin{lemma}\label{L:reductive}
	The Lie pseudo-group $\bar{\mathcal{G}}$ acts reductively.
\end{lemma}

\begin{proof}
	
	Given a point $p_0\in M$ and a frame $u_0$ on it, we consider 
    \[
    \mathcal{P}(u_0) = \left\{ f_* \circ u_0 \in \LM :\: f \in \bar{\mathcal{G}}\right\}
    \]
    and $\mathcal{P}(u_0, \mathcal{G}\cdot p_0) = \{ f_* \circ u_0 \in \mathcal{L}(\mathcal{G}\cdot p_0) :\: f \in \mathcal{G}\}$, the reduction of the frame bundle of $M$ and the orbit $\mathcal{G}\cdot p_0$, respectively. Since every locally homogeneous Riemannian manifold is reductive (see Theorem~\ref{Loc Hom => Reductive} above), we have the reductive decomposition in the leaves: $T_{u_0} \mathcal{P}(u_0, \mathcal{G}\cdot p_0) = \h + \m$, where $\m$ is an $\Ad(H_{\G} (p_0))$-invariant subspace. 
    
    We consider the decomposition $T_{u_0} \mathcal{P}(u_0) = \h + (\a + \m)$, where 
    \[
    \a = \left\{ \dt (\varphi)_* \circ u_0 :\: \varphi \in \mathcal{A} \right\}.
    \]
    As every unit vector field $\xi$ orthogonal to the leaves is invariant under the local transformations of $\mathcal{G}$, the flow of $\xi$ commutes with every $f\in \mathcal{G}$. Consequently, $\a$ is $\Ad(\mathrm{H}_{\G} (p_0))$-invariant and $\Ad (\varphi_{t_0} \circ L_{g^{-1}})$-invariant. Finally, as $\varphi_{t_0} \circ L_{g^{-1}}$ preserves $\xi$ and the horizontal distribution defined by the connection above, the subspace $\m$ is $\Ad (\varphi_{t_0} \circ L_{g^{-1}})$-invariant.
\end{proof}

Therefore, from~\cite[Thm.~3.7]{CC2022*}, there exists a global connection $\TT$ such that 
\begin{equation*}
	\TT \curv = 0, \quad \TT \tors = 0.
\end{equation*}
Finally, the proof of
\begin{equation*}
	 \TT \xi = 0, \quad \TT_X g = 0, \quad \tors(X,Y) \in \mathcal{D}
\end{equation*}
is analogous to the corresponding proof in Theorem~\ref{CO1.teorema principal}.\renewcommand*{\qedsymbol}{\(\blacksquare\)}
\end{proof}

We also have the following properties of CO1-AS-manifolds.

\begin{proposition} \label{prop:properties}
Let $(M,g, \TT)$ be a CO1-AS-manifold and let $S = \T - \TT$ be the corresponding cohomogeneity one structure tensor. For every $X$, $Y \in \mathcal{D}$, we have
\begin{enumerate}
	\item $\TT_X S = 0$,
	\item $S_X \cdot g=0$,
	\item $g(S_XY,\xi) = g(\T_XY,\xi)=g(\mathrm{II}(X,Y),\xi)$, 
	\item $S_{\xi} \xi = 0$,
\end{enumerate}
where $\mathrm{II}$ is the second fundamental form of the leaves.
\end{proposition}
\begin{proof}
\textbf{1.} Since $\TT_X g = 0$, every parallel transport defined by $\TT$ along curves belonging to one leaf preserves the Levi-Civita connection and the connection $\nabla$. Therefore, they also preserve the difference tensor $S = \T - \TT$, that is, $\TT_X S = 0$.

\textbf{2.} From $\TT_X g = 0$ and $\T g = 0$, we have $(\T -\TT)_X g = S_X g = 0$.

\textbf{3.} First, $g(S_XY,\xi) = g((\T - \TT)_XY,\xi)=g(\T_XY,\xi) -g( \TT_XY,\xi)$. From $\TT_X g = 0$, this is equal to $g(\T_XY,\xi) + g( Y, \TT_X \xi)= g((\T_XY)^{\perp},\xi) =g(\mathrm{II}(X,Y),\xi)$.

\textbf{4.} It is a consequence of the fact that $\xi$ is a geodesic vector field and $\TT \xi = 0$. 
\end{proof}


\section{Decomposition of cohomogeneity one Structures} \label{s:decom-CO1-AS}


We now explore the infinitesimal models associated with CO1-AS manifolds as well as their classification.

Let $(V,g)$ be a vector space of dimension $n+1$, $n\in\N$,  endowed with a positive definite inner product $g$ and a unit vector $\xi$. Let
\begin{equation*}
\curv \colon V \wedge V \to \mathrm{End}(V), \quad \tors \colon V \to \mathrm{End}(V),
\end{equation*}
be two linear homomorphisms. We say that $(\curv,\tors)$ is an  \textit{infinitesimal cohomogeneity one model} if it satisfies
\begin{equation*}
\begin{alignedat}{111}
	&\tors_XY +\tors_YX = 0,							\\
	&\curv_{XY}Z + \curv_{YX}Z = 0,					\\
	&\curv_{XY} \cdot \tors = \curv_{XY} \cdot \curv = 0,\\
	&\SC{XYZ} \curv_{XY}Z + \tors_{\tors_XY}Z = 0,	\\
	&\SC{XYZ} \curv_{\tors_XY Z} = 0,		\\
	&\curv_{XY}\cdot \xi = 0, \quad   
\end{alignedat}
\end{equation*}
and, for all $X$, $Y \in D =  \mathcal{L}(\xi)^\perp$, 
\begin{equation*}
\begin{alignedat}{111}
	&\curv_{XY}\cdot g = 0,  \\
	&\curv_{XY}\cdot S = 0,  \\
	&\tors_{X}Y \in D,            
\end{alignedat}
\end{equation*}
where $\SC{XYZ}$ is the cyclic sum, and the dot stands for the action of $\curv_{XY}$ on $V$ as a derivation.


 Two CO1-AS infinitesimal models $(V, \tors, \curv, \xi, g, S)$, $(V, \tors ' ,\curv', \xi ', g', S')$ are said to be isomorphic if there is an isometry
\begin{equation*}
    f \colon (V,g) \to (V',g')
\end{equation*}
such that,
\begin{equation*}
    f\, \curv = \curv', \quad f\, \tors = \tors', \quad f (\xi) = \xi , \quad f\, S = S'.
\end{equation*}

\begin{remark}
In particular, at any point $p$ of a CO1-AS manifold $M$ there is an infinitesimal cohomogeneity one model by taking $V = T_p M$, $g_p$, $\xi_p$, $\curv_p$ and $\tors_p$. According to Theorem~\ref{CO1.teorema local}, two points $p$ and $p'$ in the same leaf of the foliation of $M$ define isomorphic infinitesimal cohomogeneity one models.
\end{remark}

Given an infinitesimal cohomogeneity one
model, we define 
\begin{equation*}
\mathcal{S}(V) =\{ S \in V^* \otimes \gl(V) : \: S_X \cdot g = 0,\, S_{\xi} \xi = 0, \, \forall \, X\in D\}
\end{equation*}
as the \emph{space of infinitesimal cohomogeneity one structures}. 
From the decomposition $V=U\oplus D$ as irreducible $\SO(n)$-submodules, $U=\mathcal{L}(\xi)$, $D=U^\perp$, we have
$$
\mathcal{S}(V)=\mathcal{S}_D(V) + \mathcal{S}_U(V)
$$
with 
\begin{align*}
\mathcal{S}_D(V) &{}=\{ S \in D^* \otimes \gl(V) :\: S_X \cdot g = 0\} \simeq D^* \otimes \so(V)=D^*\otimes V\wedge V,\\
\mathcal{S}_U(V) &{}=\{ S \in U^* \otimes \gl(V) :\: S_\xi \xi = 0 \}.
\end{align*}
We now explore each of these $\SO(n)$-submodules. As usual, we can identify spaces and their duals through the metric $g$.

We begin with $\mathcal{S}_D(V)$. It can be decomposed in  $\SO(n)$-submodules as
\begin{align*}
\mathcal{S}_D(V) &{}= D^* \otimes V^*\wedge V^* \\
&{}= D^* \otimes D^* \wedge D^* + D^* \otimes U^* \otimes D^*,
\end{align*}
decomposition that is denoted by
\begin{align*}
\mathcal{T}(V) &{}=  D^* \otimes D^* \wedge D^* ,\\
\mathcal{II}(V) &{}=  D^* \otimes U^* \otimes D^*.
\end{align*}
As we will see when we understand these structures coming from a CO1-AS manifold, this notation is motivated by the fact that $\mathcal{T}(V)$ is the space of all possible infinitesimal homogeneous structure tensors of the leaves and $\mathcal{II}(V)$ is the space of all possible second fundamental forms of the leaves at each point. The decomposition of these modules into irreducible submodules can be derived from the results in the literature. They are as follows.


\begin{proposition}
The submodule $\mathcal{T}(V)$ decomposes as $\SO(n)$-irreducible submodules as,
\begin{equation*}
	\mathcal{T}(V)=\mathcal{T}_1(V)+\mathcal{T}_2(V)+\mathcal{T}_3(V),
\end{equation*}
with explicit expressions,
\begin{equation*}
	\begin{alignedat}{111}
		\mathcal{T}_1(V) &  = \Bigl\{ S \in \mathcal{T}(V) \,:\,S_{XYZ}=g(X,Y)\theta(Z)-g(X,Z)\theta(Y), \,\theta\in D^* \Bigr\},\\
		\mathcal{T}_2(V) & = \Bigl\{ S \in \mathcal{T}(V) \,:\,\SC{XYZ} S_{XYZ}=0, \,c_{12}(S)=0\Bigr\},\\
		\mathcal{T}_3(V) & = \Bigl\{ S \in \mathcal{T}(V) \,:\,S_{XYZ}+S_{YXZ}=0 \Bigr\},
	\end{alignedat}
\end{equation*}
where $c_{12}(S)(Z)=\sum_{i=1}^{n}S_{e_ie_iZ}$ for any orthonormal basis $\{e_1,\ \ldots\ ,e_{n}\}$ of $D$. 
\end{proposition}

\begin{proof}
    The proof follows directly from~\cite[Thm.~3.1]{TV1983} applied to our notation. 
\end{proof}


\begin{proposition} \label{thm:II(V)}
The submodule $\mathcal{II}(V)$ decomposes as $\SO(n)$-irreducible submodules as,
\begin{equation*}
	\mathcal{II}(V)=\mathcal{II}_1(V)+\mathcal{II}_2(V)+\mathcal{II}_3(V),
\end{equation*}
with explicit expressions,
\begin{equation*}
	\begin{alignedat}{1}
		\mathcal{II}_1(V) &  = \Bigl\{ S \in \mathcal{II}(V):\: S_{X Y Z} =  \lambda g(X,Z)g(Y,\xi) , \,\lambda\in \R \Bigr\},\\
		\mathcal{II}_2(V) & = \Bigl\{ S \in \mathcal{II}(V):\: S_{XYZ} = g(Y,\xi) S_{X\xi Z}, \, S_{X\xi Z}= S_{Z \xi X}, \, c_{13}(S)=0 \Bigr\},\\
		\mathcal{II}_3(V) & = \Bigl\{ S \in \mathcal{II}(V):\: S_{XYZ} = g(Y,\xi) S_{X\xi Z}, \, S_{X\xi Z}= -S_{Z \xi X} \Bigr\},
	\end{alignedat}
\end{equation*}
where $c_{13}(S)(Y)=\sum_{i=1}^{n}S_{e_i Y e_i}$ for any orthonormal basis $\{e_1,\ \ldots\ ,e_{n}\}$ of $D$. 
\end{proposition}

\begin{proof}
    Recall that $\mathcal{II}(V) = D^* \otimes U^* \otimes D^*$. Indeed, for any $S\in D^*\otimes U^*\otimes D^*$, we have $S_{XYZ} = g(Y,\xi) S_{X \xi Z}$, and $S_{\cdot \ \xi \ \cdot }\in D^*\otimes D^*$. This space decomposes into irreducible $\mathrm{SL} (D)$ submodules as
    \begin{equation*}
        \mathcal{II}(V) = D^* \wedge D^* + S^2 D^*, 
    \end{equation*}
    in symmetric and skew-symmetric endomorphisms, respectively. To get the $\SO(D) \subset \mathrm{SL}(D)$ irreducible submodules we have to take traces in the endomorphisms, i.~e.,
    \[
        \mathcal{II}(V) = D^* \wedge D^* + \{ S_{X\xi Z}\in S^2 D^*:\: \mathrm{c}_{12}(S)(\xi) = 0\} + \{S_{X\xi Z} \in S^2 D^*:\: \mathrm{c}_{12}(S)(\xi) \neq 0\}
    \]
    and this decomposition into irreducible submodules yields the three expressions $\mathcal{II}_3(V)$, $\mathcal{II}_2(V)$ and $\mathcal{II}_1(V)$, respectively. 
\end{proof}


We now analyse
$$ \mathcal{S}_U(V) =\{ S \in U^* \otimes \gl(V) :\: S_\xi \xi = 0 \} = U^* \otimes V^* \otimes V.
$$
From $V=U\oplus D$ we get
$$U^* \otimes V^* \otimes V^* = U^* \otimes U^* \otimes V^* + U^* \otimes D^* \otimes D^* + U^* \otimes D^* \otimes U^*,
$$
and since $S_\xi \xi=0$, the space $   \mathcal{S}_U(V) $
is exactly 
$$
\mathcal{S}_U(V) =U^* \otimes D^* \otimes D^* + U^* \otimes D^* \otimes U^*,
$$
decomposition that is denoted by
\begin{align*}
\mathcal{Z}(V)&= U^* \otimes D^* \otimes D^* ,\\
S_U ^1(V) &= U^* \otimes D^* \otimes U^*.
\end{align*}


\begin{proposition}
The submodule $\mathcal{Z}(V)$ decomposes as $\SO(n)$-irreducible submodules as,
\begin{equation*}
	\mathcal{Z}(V)=\mathcal{Z}_1(V)+\mathcal{Z}_2(V)+\mathcal{Z}_3(V),
\end{equation*}
with explicit expressions,
\begin{equation*}
	\begin{alignedat}{111}
		\mathcal{Z}_1(V) &  = \Bigl\{ S \in \mathcal{Z}(V):\: S_{X Y Z} = \lambda g(Y,Z)g(X,\xi), \, \lambda \in \R \Bigr\},\\
		\mathcal{Z}_2(V) & = \Bigl\{ S \in \mathcal{Z}(V):\: S_{XYZ} = g(X,\xi) S_{\xi Y Z}, \,S_{\xi Y Z}=S_{ \xi YX}, \,c_{23}(S)=0 \Bigr\},\\
		\mathcal{Z}_3(V) & = \Bigl\{ S \in \mathcal{Z}(V):\: S_{XYZ} = g(X,\xi) S_{\xi Y Z}, \,S_{\xi Y Z}= - S_{ \xi ZY}\Bigr\},
	\end{alignedat}
\end{equation*}
where $c_{23}(S)(X)=\sum_{i=1}^{n}S_{Xe_ie_i}$ for any orthonormal basis $\{e_1,\ \ldots\ ,e_{n}\}$ of $D$.
\end{proposition}

\begin{proof}
    As $\mathcal{Z}(V) \simeq \mathcal{II}(V)$, this decomposition is analogous to Theorem~\ref{thm:II(V)} doing a permutation in the indices. 
\end{proof}


\begin{proposition}
    The submodule $S_U ^1 (V)$ is $\SO(D)$-irreducible and its explicit expression is,
    \begin{equation*}
        \mathcal{S}_U ^1(V) = \Bigl\{ S\in \mathcal{S}(V):\:  S_{XYZ} = g(Y, \eta) g(X,\xi)g(Z,\xi), \, \eta \in D \Bigr\}.
    \end{equation*}
\end{proposition}
\begin{proof}
    It is direct from the fact $S_U ^1 (V) = U^* \otimes D^* \otimes U^* \simeq D^*$ and it is irreducible. Its explicit expression is direct.
\end{proof}


Therefore, we can write the space of cohomogeneity one structures as
\begin{equation*}
    \mathcal{S}(V) = \mathcal{T}(V) + \mathcal{II} (V) + \mathcal{Z}(V) + S_U ^1 (V).
\end{equation*}
From the previous four propositions, $\mathcal{S}(V)$ decomposes into the sum of ten irreducible $\mathrm{SO}(n)$-submodules. The local isometries in CO1-AS manifolds in Theorem~\ref{CO1.teorema local} give the following result.

\begin{proposition}
If the cohomogeneity one structure $S_p$ of a CO1-AS manifold $(M,g)$ at a point $p\in M$ belongs to a certain irreducible submodule or sum of irreducible submodules, then $S_{p'}$ belongs to the same submodule or sum of submodules for any other point $p'$ in the leaf of $p$.
\end{proposition}

We now present a few simple geometric results stemming from the classification above. It is easy to see that the  projections to each submodule are
\begin{align*}
    \Pi_{\mathcal{T}(V)} (S)_{XYZ} &= \sum _{i,j =0} ^{n} \ g(X, e_i) g(Y, e_j) S_{e_i e_j Z}, \\
    \Pi_{\mathcal{II}(V)} (S)_{XYZ} &= \sum _{i =0} ^{n} \ g(X, e_i) g(Y, \xi) S_{e_i \xi Z}, \\
    \Pi_{\mathcal{Z}(V)} (S)_{XYZ} &= \sum _{i =0} ^{n} \ g(Y, e_i) g(X, \xi) S_{e_i \xi Z}, \\[1ex]
    \Pi_{S_U ^1(V)} (S)_{XYZ} &= g(X, \xi) g(Z, \xi) S_{ \xi Y \xi}.
\end{align*}

\begin{proposition}
Let $(M,g,\xi)$ be a CO1-AS manifold and $p\in M$. 

If $\Pi_{\mathcal{T}(V)} (S_p)_{XYZ}= 0$, then the leaf of $p$ is a locally symmetric Riemannian manifold.

If $\Pi_{\mathcal{T}(V)} (S_p)_{XYZ} \in \mathcal{T}_1(V)$, then the leaf of $p$ is locally isometric to the real hyperbolic space.
 
If $\Pi_{\mathcal{II}(V)} (S_p)_{XYZ} = 0$, then the leaf of $p$ is a totally geodesic submanifold. 
\end{proposition}

\section{The canonical cohomogeneity one structure} \label{s:difference-tensor}

The proof of (1) implies (2) of the classical Ambrose-Singer Theorem, makes use of the so-called canonical connection associated with the given reductive decomposition. This canonical connection can also be characterized algebraically (see~\cite[p.~43, Thm.~8.1]{N1954}) and geometrically  (see~\cite[p.~192, Prop.~2.4]{KN1969}).  We now analyze the connection constructed in the proof (1) implies (2) of Theorem~\ref{CO1.teorema principal} to get a similar concept in cohomogeneity one manifolds. 

We consider a reductive homogeneous space $M=G/H$ (not necessarily Riemannian), with reductive decomposition $\mathfrak{g}=\mathfrak{h}+\mathfrak{m}$.
For any $X\in\mathfrak{g}$, we denote by $X^*$ the fundamental vector field associated with $X$. The canonical connection $\tilde{\nabla}$ of $M$ with respect to the above decomposition is determined by (cf.~\cite[p.~188-192]{KN1969})
\[
\bigl(\tilde{\nabla}_{X^*}Y^*\bigr)_p= - \tors_p(Y_p^*,X_p^*) = \tors_p(X_p^*,Y_p^*),
\]
where $\tors$ is the torsion of $\TT$, for any $X$, $Y \in \m$, $p=eH$. From~\cite[Ch.~X, Thm.~2.6]{KN1969}, we also have $\tors_p(X_p^*,Y_p^*)= -[X,Y]_\mathfrak{m}$  and, therefore,
\[
\bigl(\tilde{\nabla}_{X^*}Y^*\bigr)_p=-[X,Y]_{\mathfrak{m}},
\]
for $X$, $Y \in \m$.

\begin{lemma}\label{lem:tildenabla-at-o}
    If $X\in\mathfrak{m}$ and $Y\in\mathfrak{g}$ then $\tilde{\nabla}_{X}Y^*=-[X,Y]_\mathfrak{m}$.
\end{lemma}

\begin{proof}
    Using the formulas for $\tilde{\nabla}$ and the torsion $\tors_p$, when the corresponding vectors are in $\mathfrak{m}$, we have
    \begin{align*}
        \tilde{\nabla}_X Y^*
        &{}=\tilde{\nabla}_{X^*_p}Y^*=\bigl(\tilde{\nabla}_{X^*}Y^*\bigr)_p\\
        &{}=\bigl(\tilde{\nabla}_{Y^*}X^*\bigr)_p+[X^*,Y^*]_p+\tors_p(X^*,Y^*)\\
        &{}=\tilde{\nabla}_{Y^*_p}X^*-[X,Y]^*_p+\tors_p(X^*_p,Y^*_p)\\
        &{}=\tilde{\nabla}_{Y_\mathfrak{m}}X^*-[X,Y]_\mathfrak{m}+\tors_p(X,Y_\mathfrak{m})\\
        &{}=-[Y_\mathfrak{m},X]_\mathfrak{m}-[X,Y]_\mathfrak{m}-[X,Y_\mathfrak{m}]_\mathfrak{m}=-[X,Y]_\mathfrak{m},
    \end{align*}
    as we wanted to show.
\end{proof}

Now we calculate $\tilde{\nabla}$ at any other point.

\begin{lemma}\label{prop:tildenabla-at-go}
    If $X$, $Y\in\mathfrak{m}$ and $f\in G$, we have 
    \[
        \bigl(\tilde{\nabla}_{X^*}Y^*\bigr)_{f(p)}
        =-f_{*p}\bigl[\bigl(\Ad(f^{-1})X\bigr)_\mathfrak{m},
        \Ad(f^{-1})Y
        \bigr]_\mathfrak{m}.
    \]
\end{lemma}

\begin{proof}
Let $X\in\mathfrak{g}$. We consider $\psi_t=\exp(tX)$, the $1$-parameter group generated by $X$. We denote by $I_{f^{-1}}$ the conjugation by $f^{-1}$, and by $\Ad(f^{-1})=I_{f^{-1}*}$ the corresponding adjoint representation of $\mathfrak{g}$. 
We have
\[
X^*_{f(p)}=\dt \psi_t(f(p))
=\dt fI_{f^{-1}}(\psi_t)(p)=f_{*p}\bigl(\Ad(f^{-1})X\bigr)_p.
\]

Hence, $\bigl(f_*^{-1}X^*\bigr)_p=\bigl(\Ad(f^{-1})X\bigr)_p$. Since $\tilde{\nabla}$ is $G$-invariant, Lemma~\ref{lem:tildenabla-at-o}
yields
\begin{align*}
\bigl(\tilde{\nabla}_{X^*}Y^*\bigr)_{f(p)}
&{}=f_{*p}\bigl(\tilde{\nabla}_{f_*^{-1}X^*}f_*^{-1}Y^*\bigr)_p
=f_{*p}\tilde{\nabla}_{\bigl(\Ad(f^{-1})X\bigr)_{\mathfrak{m}}}\bigl(\Ad(f^{-1})Y\bigr)^*\\
&{}=f_{*p}\bigl[\bigl(\Ad(f^{-1})X\bigr)_{\mathfrak{m}},\Ad(f^{-1})Y
\bigr]_\mathfrak{m},
\end{align*}
as claimed.
\end{proof}

Now we return to our original problem.
We consider a Riemannian manifold $M$ that is being acted upon by a group of isometries $G$ of cohomogeneity one, all whose orbits are principal.
As we have seen, we can define a global unit vector field $\xi$ that is orthogonal to all the orbits of $G$.
We denote by $\varphi_t$ the flow of $\xi$.
Then $\varphi_t$ determines a $1$-parameter group $A$ that is isomorphic to $\mathbb{R}$ or $\mathbb{S}^1$.
As a consequence, $M$ is now a homogeneous space acted upon by $A\times G$ by the formula $(t,g)\cdot p=\varphi_t(g\cdot p)=g\cdot\varphi_t(p)$.
In particular, $A$ is contained in the center of $A\times G$, and $\xi$ is the fundamental vector field on $M$ determined by $1\in A$.

From now on, we fix $p\in M$.
We have the reductive decomposition $\mathfrak{a}+\mathfrak{m}=\mathfrak{h}+(\mathfrak{a}+\mathfrak{m})$ at $p$.
We determine the canonical connection $\tilde{\nabla}$ associated with this decomposition.
We define $\gamma(t)=\varphi_t(p)$, which is a unit speed geodesic that intersects all the orbits of $G$ orthogonally.

\begin{proposition}\label{prop:C1-canonical}
If $X$, $Y\in\mathfrak{m}$, then we have
$\bigl(\tilde{\nabla}_\xi X^*\bigr)_{\gamma(t)}=\bigl(\tilde{\nabla}_{X^*} \xi\bigr)_{\gamma(t)}
=\bigl(\tilde{\nabla}_\xi \xi\bigr)_{\gamma(t)}=0$, and
$\bigl(\tilde{\nabla}_{X^*} Y^*\bigr)_{\gamma(t)}=-\varphi_{t*p}[X,Y]_\mathfrak{m}$.
\end{proposition}

\begin{proof}
Since $A$ is contained in the center of $A\times G$ we have that $\Ad(t)$, $t\in A$, acts as the identity on $\mathfrak{a}+\mathfrak{g}$.
Thus, $\Ad(t^{-1})\xi=\xi$ and $\Ad(t^{-1})X=X$ for all $X\in\mathfrak{m}$.
Since $[\mathfrak{a},\mathfrak{a}+\mathfrak{g}]=0$, the first three equalities of the statement follow from Lemma~\ref{prop:tildenabla-at-go}.
Finally, if $X$, $Y\in\mathfrak{m}$, Lemma~\ref{prop:tildenabla-at-go} yields
\begin{align*}
\bigl(\tilde{\nabla}_{X^*}Y^*\bigr)_{\gamma(t)}
&{}=-\varphi_{t*p}\bigl[\bigl(\Ad(t^{-1})X\bigr)_{\mathfrak{a}+\mathfrak{m}},
\Ad(t^{-1})Y
\bigr]_{\mathfrak{a}+\mathfrak{m}}\\
&{}=-\varphi_{t*p}\bigl[X_{\mathfrak{a}+\mathfrak{m}},
Y
\bigr]_{\mathfrak{a}+\mathfrak{m}}
=-\varphi_{t*p}\bigl[X,Y\bigr]_{\mathfrak{m}},
\end{align*}
which finishes the proof.
\end{proof}

Finally, we calculate the cohomogeneity one structure $S=\nabla-\tilde{\nabla}$, where $\nabla$ is the Levi-Civita connection of $M$.

\begin{proposition}\label{prop:difference}
Let $X$, $Y$, $Z\in\mathfrak{m}$, and $a$, $b$, $c\in\mathbb{R}$.
Then, the difference tensor $S$ is given by
\begin{align*}
&2g\big( S_{a\xi+X^*}(b\xi+Y^*),c\xi+Z^*\big)_{\gamma(t)}={}\\
&\qquad{}a\frac{d}{dt}g\big(\varphi_{t*p}Y,\varphi_{t*p}Z\big)
+b\frac{d}{dt}g\big(\varphi_{t*p}X,\varphi_{t*p}Z\big)
-c\frac{d}{dt}g\big(\varphi_{t*p}X,\varphi_{t*p}Y\big)\\[1ex]
&\qquad{}+g\big(\varphi_{t*p}[X,Y]_\mathfrak{m},\varphi_{t*p}Z\big)
-g\big(\varphi_{t*p}[X,Z]_\mathfrak{m},\varphi_{t*p}Y\big)
-g\big(\varphi_{t*p}[Y,Z]_\mathfrak{m},\varphi_{t*p}X\big).
\end{align*}
\end{proposition}

\begin{proof}
Since $\xi$ is a geodesic vector field, $\nabla_\xi\xi=0$.
Moreover, since $Y^*$, $Y\in\mathfrak{m}$, is Killing, we have $g\big(\nabla_\xi Y^*,\xi\big)=0$ and $g\big(\nabla_\xi Y^*,Z^*\big)=-g\big(\nabla_{Z^*}Y^*,\xi\big)$.
Also, $g\big(\xi,\xi\big)=1$ implies $g\big(\nabla_{X^*}\xi,\xi\big)=0$, and $g\big( X^*,\xi\big)=0$ implies $g\big(\nabla_{X^*}\xi,Z^*\big)=-g\big(\nabla_{X^*}Z^*,\xi\big)$.
Altogether this gives
\begin{equation}\label{eq:S}
\begin{aligned}
2g\big( \nabla_{a\xi+X^*}(b\xi+Y^*),c\xi+Z^*\big)={}
&-2ag\big(\nabla_{Z^*}Y^*,\xi\big)
-2bg\big(\nabla_{X^*}Z^*,\xi\big)\\
&+2cg\big(\nabla_{X^*}Y^*,\xi\big)+2g\big(\nabla_{X^*}Y^*,Z^*\big).
\end{aligned}
\end{equation}

We calculate $g\big(\nabla_{X^*}Y^*,\xi\big)_{\gamma(t)}$ for $X$, $Y\in\mathfrak{m}$.
First note that $X^*_{\gamma(t)}=X^*_{\varphi_t(p)}=\varphi_{t*p}X^*_p$ because $\varphi_t$ commutes with the action of $G$.
The fact that $\xi$ is geodesic, $X^*$ and $Y^*$ are Killing vector fields, and the Levi-Civita connection is torsion-free yields
\begin{align*}
\frac{d}{dt}g\big(\varphi_{t*p}X,\varphi_{t*p}Y\big)
&{}=\frac{d}{dt}g\big( X^*_{\gamma(t)},Y^*_{\gamma(t)}\big)\\
&{}=g\big(\nabla_{\xi_{\gamma(t)}}X^*,Y^*_{\gamma(t)}\big)
+g\big( X^*,\nabla_{\xi_{\gamma(t)}}Y^*_{\gamma(t)}\big)\\
&{}=-g\big(\nabla_{Y^*_{\gamma(t)}}X^*,\xi_{\gamma(t)}\big)
-g\big(\nabla_{X^*_{\gamma(t)}}Y^*,\xi_{\gamma(t)}\big)\\
&{}=-g\big( 2\nabla_{X^*_{\gamma(t)}}Y^*-[X^*,Y^*]_{\gamma(t)},\xi_{\gamma(t)}\big)\\
&{}=-2g\big( \nabla_{X^*_{\gamma(t)}}Y^*,\xi_{\gamma(t)}\big).
\end{align*}

The last addend of~\eqref{eq:S} can be calculated from~\cite[7.27]{B1987} as
\[
g\big(\nabla_{X^*}Y^*,Z^*\big)=g\big([X^*,Y^*],Z^*\big)+g\big([X^*,Z^*],Y^*\big)+g\big([Y^*,Z^*],X^*\big).
\]
Since $X^*_{\varphi_t(p)}=\varphi_{t*p}X_p$, we see that $X^*$ is $\varphi_t$-related to itself.
Thus, $[X^*,Y^*]_{\gamma(t)}=[X^*,Y^*]_{\varphi_t(p)}=\varphi_{t*p}[X^*,Y^*]_p=-\varphi_{t*p}[X,Y]_\mathfrak{m}$.

Substituting in~\eqref{eq:S} we obtain the Levi-Civita connection of $M$ along $\gamma$.
This, together with Proposition~\ref{prop:C1-canonical}, finishes the proof of Proposition~\ref{prop:difference}.
\end{proof}

\begin{remark}
Let $A_{\gamma(t)}$ denote the shape operator of $G\cdot\gamma(t)$ at $\gamma(t)$ with respect to the normal vector $\xi_{\gamma(t)}$, that is, $A_{\gamma(t)}(v)=-\nabla_v\xi$, $v\in T_{\gamma(t)}\bigl(G\cdot \gamma(t)\bigr)$.
Then,
\[
g\big( A_{\gamma(t)}X^*_{\gamma(t)},Y^*_{\gamma(t)}\big)
=-g\big(\nabla_{X^*}Y^*,\xi\big)_{\gamma(t)}
=\frac{1}{2}\frac{d}{dt}g\big(\varphi_{t*p}X,\varphi_{t*p}Y\big).
\]

Furthermore, $X^*_{\gamma(t)}$ is a Jacobi vector field along $\gamma$ with $\frac{d}{dt}_{\vert_0}X^*_{\gamma(t)}=-A_p X^*_p$.
Thus, if one is able to solve the Jacobi equation along $\gamma$, then $X^*_{\gamma(t)}$ can be calculated explicitly, and so can the shape operator $A_{\gamma(t)}$.
\end{remark}

\begin{definition}
    The cohomogeneity one structure $S$ given in Proposition~\ref{prop:difference} is called \emph{the canonical cohomogeneity one structure}.
\end{definition}
From Proposition~\ref{prop:C1-canonical}, the canonical cohomogeneity one structure satisfies
\begin{equation*}
        \tors (\xi, \ \cdot \ ) = 0,
\end{equation*}
    where $\tors$ is the torsion of $\TT = \T - S$.

\begin{proposition}
    Let $S_1$ and $S_2$ be two cohomogeneity one structures on $(M,g)$ associated with $\xi$ such that $\tors_1 (\xi, \ \cdot \ ) = 0$ and $\tors_2 (\xi, \ \cdot \ ) = 0$. If $\Pi_{\mathcal{T}(V)} (S_1)  = \Pi_{\mathcal{T}(V)} (S_2)$, then $S_1 = S_2$.
\end{proposition}

\begin{proof}
    We notice that the distribution $\mathcal{D} = \{ X : \: g(X, \xi) = 0 \}$ is independent of the choice of $S_1$ and $S_2$. Therefore, it makes sense to consider $\Pi_{\mathcal{T}(V)} (S_1)  = \Pi_{\mathcal{T}(V)} (S_2)$. This means that $(S_1)_X Y = (S_2)_X Y$ for all $X$, $Y \in \mathcal{D}$. For $i=1$, $2$, as the Levi-Civita connection is torsion-free,
    \[
    \tors_i (\xi, A) = - (S_i)_{\xi} A + (S_i)_A \xi, \quad \forall A\in TM.   
    \]
    We apply $\tors_i (\xi, \cdot) = 0$ and obtain $(S_i)_{\xi} A = (S_i)_A \xi$. For all $a$, $b\in \R$, we apply statements (1) to (4) in Proposition~\ref{prop:properties},
    \begin{align*}
        g((S_i)_{a\xi +X} \xi , c\xi +Y) &{}\overset{(4)}{=} g((S_i)_{X} \xi , c\xi +Y)\\
        & {}\overset{(2)}{=} g((S_i)_{X} \xi , Y)\\
        & {}\overset{(3)}{=} - \mathrm{II}(X,Y),
    \end{align*}
    which is independent of the choice of $i$. Then $S_1 = S_2$.    
\end{proof}

\begin{remark}
    The canonical cohomogeneity one structure $S$ is the unique (up to homogeneous structure in the leaves) cohomogeneity one structure satisfying, 
    \begin{equation*}
        \tors (\xi, \ \cdot \ ) = 0,
    \end{equation*}
    where $\tors$ is the torsion of $\TT = \T - S$.
\end{remark}

\section{Examples}\label{s:examples}

\subsection*{Parallel hyperplanes in Euclidean spaces}

Let $(\R ^{n}, g_{\R^n})$ be the Euclidean space. If $H$ is a closed subgroup of $\SO (n-1)$, then $G= H \ltimes \R ^{n-1}$ acts transitively on $(\R ^{n-1}, g_{\R^{n-1}})$. Consequently, $G$ acts with cohomogeneity one on $\R^{n}$. For $p = (p_1, \ \ldots \ , p_n)$, the orbit space $G \cdot p = \{ (x_1, \ \ldots \ x_{n-1}, p_n) : \: x_i \in \R \}$ gives a foliation without singular orbits. The unit vector field orthogonal to the leaves is $\xi _p = (0, \ \ldots \ 0,1)$ and its flow is $\varphi_t (p) = (p_1, \ \ldots \ , p_n +t)$. Since $\varphi_t$ is an isometry and $\R^{n-1}$ is an abelian ideal of $\R^{n}$, the first and second rows of Proposition~\ref{prop:difference} are zero. Thus, $S = 0$.

\subsection*{Concentric spheres in Euclidean spaces}
The space $(\R ^l - \{0\}, g_{\R^l})$ equipped with the Euclidean metric is isometric to the warped product $(\R^+ \times_f \mathbb{S}^{l-1},g)$, $g=dr^2 + f(r)^2 g_{\mathbb{S}^{l-1}}$, where $r$ is the distance to the origin, $f(r)=r$, and $g_{\mathbb{S}^{l-1}}$ is the round metric of the sphere. Let $\pi _2\colon\R^+ \times_f \mathbb{S}^{l-1} \to \mathbb{S}^{l-1}$ be the projection onto the second factor. Consider a transitive Lie group $G$ of isometries of $(\mathbb{S}^{l-1},g_{\mathbb{S}^{l-1}})$ and the canonical connection $\TT^{\mathbb{S}^{l-1}}$ associated with it. 
We also consider the pseudo-group $(\R^+, +)$ acting transitively on $\R^+$. The Lie pseudo-group $\R^+ \times G$ acts transitively on $\R^+ \times_f\mathbb{S}^{l-1}$ and there is an AS-connection $\TT$ (non-necessarily Riemannian and non necessarily complete) such that
\begin{equation*}
	\TT \tors = 0, \qquad \TT \curv = 0.
\end{equation*}
Actually, 
\begin{equation*}
	\TT = \T ^{\R ^+} \oplus \TT ^{\mathbb{S}^{l-1}},
\end{equation*}
where $\T ^{\R ^+}$ is the Levi-Civita connection for $\R ^+$.

The action of $G$ is of cohomogeneity one on $\R^+ \times _f \mathbb{S}^{l-1} = \R^+ \times G/H$ and it gives a foliation of spheres $\mathbb{S} ^{l-1} (r)$ of radius $r > 0$. The normal vector field for each one of these submanifolds is, $\eta (r, p) = \partial/ \partial r$ (in Euclidean coordinates, $\eta(x) = \frac{x}{||x||}$) and the associated CO1-AS connection is $\TT$. In terms of the warped elements (see~\cite[p.~206, Prop.~35]{O1983} the Levi-Civita connection of $\R^+ \times _f \mathbb{S}^{l-1}$ is
\begin{equation*}
	\T = \T^{\R ^+} \oplus \T ^{\mathbb{S}^{l-1}} - \frac{f^2 \pi_2{}^* (g_{\mathbb{S}^{l-1}})}{f} \mathrm{grad} (f) + \sum_{i=1} ^{l-1} \frac{\eta (f)}{f}  (\eta^* \otimes de^i +de^i \otimes \eta^*)  \pa{e^i},
\end{equation*}
where $\T ^{\mathbb{S}^{l-1}}$ is the Levi-Civita connection for the sphere $\mathbb{S}^{l-1}$, $\left\{\pa{e^1} ,\ \dots \ ,\pa{e^{l-1}} \right\}$ is an orthonormal basis of $T_x \mathbb{S}^{l-1}$ and $\left\{ d e^1 ,\ \dots \ , de^{l-1}\right\}$ is its dual basis. As $\eta(f) = 1$, we have
\begin{equation*}
	\T = \T^{\R ^+} \oplus \T ^{\mathbb{S}^{l-1}} - \frac{f^2 \pi_2 {}^* (g_{\mathbb{S}^{l-1}})}{||x||} \eta + \frac{1}{||x||}\sum_{i=1} ^{l-1}   \left(\eta^* \otimes de^i +de^i \otimes \eta^* \right)  \pa{e^i}.
\end{equation*}

Hence, the sphere foliation of $(\R^{n-k-1} -\{0\} , g)$ induced by the Lie group $G$ possesses a cohomogeneity one structure tensor 
\begin{equation} \label{CO1.eq:co1 structure spheres}
	S =  0 \oplus S^{\mathbb{S}^{l-1}} - \frac{f^2 \pi_2 {}^* (g_{\mathbb{S}^{l-1}})}{||x||}\eta + \frac{1}{||x||} \sum_{i=1} ^{l-1} \left(\eta^* \otimes de^i +de^i \otimes \eta^* \right)  \pa{e^i},
\end{equation}
where $S^{\mathbb{S}^{l-1}} = \T^{\mathbb{S}^{l-1}} - \TT ^{\mathbb{S}^{l-1}}$ is the homogeneous structure for the action of $G$ on $\mathbb{S}^{l-1} = G/H$. In other words, 
\begin{align*}
	S_B C ={}&
    S^{\mathbb{S}^{l-1}}_{(\pi_2)_* B} (\pi_2)_*C - \frac{g((\pi_2)_* B,(\pi_2)_* C))}{r} \ \pa{r}\\
    &{}+ \frac{1}{r} \left(g\left(\pa{r}, B\right)(\pi_2)_* C + g\left(\pa{r}, C\right) (\pi_2)_*B\right).
\end{align*}

In particular, for $p\in \R^+ \times _f \mathbb{S}^{l-1}$, if we take $B$, $C \in T_p \mathbb{S}^{l-1}$, we have
\begin{align*}
	S_B C ={}&
    S^{\mathbb{S}^{l-1}}_{B} C - r g_{\mathbb{S}^{l-1}}(B,C) \eta \\
    S_B \eta ={}& \frac{1}{r} B, \quad S_{\eta} C ={} \frac{1}{r} C.
\end{align*}

In a broader sense, let $(M_1, g_2,\TT_1)$ and  $(M_2, g_2, \TT_2)$ be two Riemannian AS-manifolds, that is, 
\begin{equation*}
    \begin{array}{cccccccccccc}
         \TT_1  \curv_1 &=& 0, &&  \TT_1  \tors_1 &=& 0, && \TT_1 g_1 &=& 0, \\
    \TT_2 \curv_2 &=& 0, && \TT_2 \tors_2 &=& 0, && \TT_2 g_2 &=& 0,
    \end{array}
\end{equation*}
where $\curv_i$ and $\tors_i$ are the curvature and torsion of $\TT_i$ with $i = 1,2$. On the one hand, the product manifold $M_1 \times M_2$
is a general AS-manifold with AS-connection
\begin{equation*}
    \TT = \TT_1 \oplus \TT_2.
\end{equation*}
On the other hand, for a positive function $f \colon M_1 \to \R^+$, we can consider the warped product $\left(M_1 \times_f M_2,\, g= \pi_1 {}^* (g_1) + (f \circ \pi_1)^2 \pi_2{}^* (g_2)\right)$ where $\pi_i$ is the projection onto $M_i$ with $i = 1$ or $2$. We write $f$ instead of $f \circ \pi_1$. In particular, if the dimension of $M_1$ is $1$, the warped product is a regular (locally) cohomogeneity one manifold and the Levi-Civita connection in local coordinates is given by (again,~\cite[p.~206, Prop.~35]{O1983})
\begin{equation*}
	\T = \T_1 \oplus \T_2 - \frac{f^2 \pi_2{}^* (g_2)}{f} \mathrm{grad} (f) + \sum_{i=1} ^{l-1} \frac{\eta (f)}{f}  (\eta^* \otimes de^i +de^i \otimes \eta^*)  \pa{e^i},
\end{equation*}
where $\eta$ is a unit vector field in $(M_1,g_1)$, $\left\{\pa{e^1} ,\ \dots \ ,\pa{e^{l-1}} \right\}$ is a local orthonormal basis of $M_2$, and $\left\{ d e^1 ,\ \dots \ , de^{l-1}\right\}$ is its dual basis. Finally, the (1, 2)-tensor
\begin{equation*}
    S = \T -\TT
\end{equation*}
is a cohomogeneity one structure.

\subsection*{The horosphere foliation in the real hyperbolic space}

The real hyperbolic space $\RH(n)$ with the hyperbolic metric  of constant scalar curvature equal to  $-1$ is isometric to the warped product $(\R\times_f \R ^{n-1}, g = dt^2 + f(t) ^2 g_{\R^{n-1}})$, with $f(t) = e^{-t}$. Following the procedure outlined above,
we can introduce the AS-connection $\TT$, defined as 
	\begin{equation*}
		\TT = \T^{\R} \oplus \T ^{\R^{n-1}},
	\end{equation*}
	where $\T^{\R }$ and $\T ^{\R^{n-1}}$ are the Levi-Civita connections of $\R$ and $\R^{n-1}$, respectively. Note that $\TT$ corresponds to the Levi-Civita connection for the Euclidean metric in $\R^{n}$. However, $\TT$ is not metric with respect to $g$, even though it satisfies
	\begin{equation*}
		\TT \curv = 0,\quad \TT \tors = 0,\quad \TT \xi = 0, \quad \TT_X g=0, \quad \tors (X,Y) \in \mathcal{L}(\xi )^{\perp},
	\end{equation*}
	where $\xi = \partial /\partial t$ and $\curv$ and $\tors$ are the curvature and torsion of $\TT$. 
    In other words, $\TT$ is an AS-connection of Theorem~\ref{CO1.teorema principal} and $S = \T - \TT$ is a cohomogeneity one structure, where $\T$ is the Levi-Civita connection of $\RH(n)$. We now compute $S$ explicitly.
	
	The Levi-Civita connection for a warped product is
	\[
		\T = \T^{\R} \oplus \T ^{\mathbb{R}^{n-1}} - \frac{f^2 \pi_2{}^* (g_{\R^{n-1}})}{f} \mathrm{grad} (f) + \sum_{i=1} ^{n-1} \frac{\xi (f)}{f}  \left(\xi^* \otimes de^i +de^i \otimes \xi^* \right)  \pa{e^i},
  \]
  that is
  \[
		\T = \T^{\R} \oplus \T ^{\mathbb{R}^{n-1}} + f^2 \pi_2 {}^* (g_{\R^{n-1}}) \ \xi  - \sum_{i=1} ^{n-1} \left(\xi^* \otimes de^i +de^i \otimes \xi^* \right)  \pa{e^i}.
    \]
	The cohomogeneity-1 structure $S = \T -\TT$ is thus
	\begin{align*}
		S &=f^2 \pi_2 {}^* (g_{\R^{n-1}}) \ \xi - \sum_{i=1} ^{l-1} \left(\xi^* \otimes de^i +de^i \otimes \xi^* \right)  \pa{e^i}\\
	   & = g((\pi_2) _* X , (\pi_2) _* Y) \xi - g(\xi, X) (\pi_2) _* Y - g(\xi, Y) (\pi_2) _* X,
	\end{align*}
and putting $\left( (\pi_2)_* X\right) = X - g(X,\xi) \xi$ and $(\pi_2)_* Y = Y - g(Y,\xi) \xi$, we finally get
\begin{equation*}
        S_X Y = g(X , Y) \xi - g(\xi, X) Y - g(\xi, Y) X + g( X,\xi) g(\xi, Y)\xi.
\end{equation*}

\subsection*{A non-canonical cohomogeneity one structure}

In all the examples above the connection $\TT$ is the canonical connection of Section~\ref{s:difference-tensor}. We now present a modification of the last example such that, even though it has the same foliation, the connection $\TT$ is not canonical.

The real hyperbolic space $\RH(n)$ is characterized by the existence of a homogeneous structure $\bar{S}$ of linear type, see \cite[Thm. 5.2]{TV1983}. This has the expression,
\[
\bar{S}_XY = g(X,Y) \xi - g(Y,\xi)X
\]
where, using the coordinates above, $\xi = \pa{t}$ and $\tilde{\bar{\nabla}}$ satisfies, 
\[
\tilde{\bar{\nabla}} \tilde{\bar{R}}= 0, \quad \tilde{\bar{\nabla}} \tilde{\bar{T}}= 0, \quad \tilde{\bar{\nabla}} \bar{S} = 0, \quad \tilde{\bar{\nabla}} g = 0,
\]
where $\bar{S} = \T - \tilde{\bar{\nabla}}$,  the Levi-Civita connection is $\T$, and the curvature and torsion of $\tilde{\bar{\nabla}}$ are  $\tilde{\bar{R}}$ and $\tilde{\bar{T}}$, respectively. We have
\[
- \tilde{\bar{T}}_A B = \bar{S}_AB -\bar{S}_BA = g(A,\xi)B - g(B,\xi)A.
\]
Then, for all $X$, $Y \in \mathcal{D} = \{ X: g(X, \xi) = 0 \}$,
\[
\tilde{\bar{T}}_X Y = 0, \quad \tilde{\bar{T}}_{\xi} Y = - Y \neq 0,
\]
which means that $\bar{S}$ is a cohomogeneity one structure, but not a canonical cohomogeneity one structure.



\printbibliography[heading=bibintoc, title=References]

\end{document}